\documentclass[graybox]{svmult}

\usepackage{mathptmx}       
\usepackage{helvet}         
\usepackage{courier}        
\usepackage{type1cm}        
\usepackage{makeidx}         
\usepackage{graphicx}        
\usepackage{multicol}        
\usepackage[bottom]{footmisc}
\usepackage{mathscinet}

\makeindex             

 \usepackage{amsmath}
 \usepackage{amssymb}
\newcommand{\N}{{\mathbb{N}}}

\begin{document}

\title*{Sums of two squares and a power}
\author{Rainer Dietmann and Christian Elsholtz}
\institute{
Rainer Dietmann\\
Department of Mathematics,
Royal Holloway, University of London,
Egham, TW20 0EX, UK\\
\email{Rainer.Dietmann@rhul.ac.uk}\\
Christian Elsholtz \at
Institut f\"ur Mathematik und Zahlentheorie, Technische Universit\"at Graz,
Kopernikusgasse 24/II, A-8010 Graz, Austria\\
\email{elsholtz@math.tugraz.at}
}
%
%

\maketitle

\begin{center}
Dedicated to the memory of Wolfgang Schwarz,\\
with admiration\\ for his broad interests, inside and outside mathematics.
\end{center}

\abstract{
We extend results of Jagy and Kaplansky and the present authors
and show that for all $k\geq 3$ there are infinitely many
positive integers $n$, which cannot be written as
$x^2+y^2+z^k=n$ for positive integers $x,y,z$,
where for $k\not\equiv 0 \bmod 4$ a congruence condition is 
imposed on $z$. These examples are of interest as there is no
congruence obstruction itself for the representation of these $n$.
This way we provide a new family of counterexamples to the Hasse 
principle or strong approximation.
}

\section{Introduction}

This paper is dedicated to the memory of Wolfgang Schwarz, who was the PhD
advisor of the second named author. In particular Wolfgang Schwarz's 
books ``Einf\"uhrung in Siebmethoden der analytischen Zahlentheorie'' 
and ``Arithmetical functions'' 
were very useful for the second author's own studies.

Looking at Schwarz's own PhD thesis, see \cite{schwarz1,schwarz2}, 
which is on sums of prime powers,
i.e. on the Goldbach-Waring problem,
 one finds a great number of results, one of those being 
the following
(Theorem 3 of \cite{schwarz2}):
For fixed $k \ge 1$ let $S_k(N)$ be the set of positive integers $n$,
with
\[
\begin{array}{ll}
  3 \leq n \leq N, &\\
  n \not\equiv 0 \bmod 2, \;  n \not\equiv 2 \bmod 3, & \mbox{for odd $k$}\\
  n \equiv 3 \bmod 24, & \mbox{for even $k$},\\
  n \not\equiv 0 \bmod 5, & \mbox{for $k \equiv 2 \bmod 4$}\\
  n \not\equiv 0,2 \bmod 5, & \mbox{for $k \equiv 0 \bmod 4$}\\
  n \not\equiv 1 \bmod p, & \mbox{for each $p \equiv 3 \bmod 4$ with $(p-1)\mid k$.}
\end{array}
\]

Then the number of integers $n\in S_k(N)$ not of the form
\[n=p_1^2 + p_2^2+p_3^k,\] 
is, for all $B>0$, at most
\[O_B\left(\frac{N}{(\log N)^B}\right).\]
This improved on a result of Hua 
\cite[Theorem 1]{Hua:1938}, who proved this with $B=\frac{k}{k+2}$.
As we had worked earlier on solutions of $x^2+y^2+z^k=n$, it is due to this 
connection that we have chosen to contribute the present note
to the volume in Memory of Wolfgang Schwarz.

As it turns out, also one of the first named author's PhD advisors worked on this kind of problem
in his PhD thesis:
without restricting the variables to primes, one should be able to obtain stronger
results, and indeed, improving on earlier work pioneered by
Davenport and Heilbronn \cite{DH} and further developed by many other authors,
Br\"udern \cite{B} has shown that there are at most
$O(N^{1-\frac{1}{k}+\epsilon})$ integers $n \leq N$ with no 
solutions of
\begin{equation}
\label{eq}
   n=x^2+y^2+z^k,
\end{equation}
where $n$ is not in a residue class excluded by congruence obstructions.
For a survey of results on sums of mixed powers see also 
\cite{BK} and \cite{VW}.

It was generally expected that for all sufficiently large $n$
the Hasse principle for equation (\ref{eq}) holds true, i.e. for all such $n$
satisfying the necessary congruence conditions there would exist
a solution of (\ref{eq}) in positive integers, see, for
example, chapter 8 in \cite{V}.
However, in 1995
Jagy and Kaplansky \cite{JK} shattered this belief
by proving that for $k=9$ and some positive constant $c$
there are at least $c \frac{N^{1/3}}{\log N}$ positive integers
$n \le N$ that are not sums of two squares and one $k$-th power.
In fact, their method works for any odd 
composite number $k$,
but not for the other cases of $k$.
In \cite{DietmannandElsholtz:2008}
we proved that a similar restriction holds for $k=4$.
That approach actually generalizes to all $k$ divisible by four
(see Theorem \ref{biquadrate}),
and by slightly modifying it we can not only get a bigger
set of exceptional $n$ but we can also handle $k$ not divisible by
four; to be more specific, we prove that
(\ref{eq}) does not satisfy 'strong approximation':
For $k \equiv  2 \pmod 4$, $k \ge 6$ and
sufficiently large $N$ we show that 
there are asymptotically
$\gg N^{1/2}/(\log N)^{1/2}$
positive integers $n \leq N$ for which equation (\ref{eq})
has no solution with $z$ fixed into a certain residue class, though
there are no congruence obstructions (see Theorem \ref{k=2t}).
For odd $k\geq 3$ we show that
there are asymptotically at least $\frac{k N^{1/k}}{2\varphi(k) \log N}$ 
such exceptional positive integers 
$n \leq N$ (see Theorem \ref{kthpower}). 

Let us further mention that Hooley \cite{Hooley:2000}
investigated sums of three squares and a $k$-th power,
Friedlander and Wooley \cite{FriedlanderandWooley:2014}
sums of two squares and three biquadrates, and Wooley \cite{Wooley:2014}
sums of squares and a `micro square', in connection with a conjecture of
Linnik. In this connection we would like to add a seemingly forgotten old
reference: Theorem 7 of Rieger \cite{Rieger:1964} states  that
the number of integers $n \leq N$ which can be written as
$n=x^2+y^2+z^k$, where $z \leq F(N)$, and $F$ is a function tending
monotonically to infinity, with $F(n) \leq \sqrt{\log N}$, is
$\gg_{k,F} \frac{N\, F(N)}{\sqrt{\log N}}$, in other words, as good as it can be.

The authors are grateful
to  Tim Browning,
J\"org Br\"udern, Roger Heath-Brown, Jan-Christoph Schlage-Puchta,
Dasheng Wei and Trevor Wooley for interesting discussions or observations.

\section{Two squares and an odd $k$-th power}
\begin{theorem}{\label{kthpower}}
Let $k\geq 3$ be odd.
Let $p$ be a prime with $p\equiv 1 \bmod 4k$.
Then there are no integers $x,y,z$, positive or negative,
with $x^2+y^2+z^k=p^{k}$ and
$z\equiv 2k \bmod 4k$.
\end{theorem}
\begin{proof}
Assume there are solutions, then
$x^2+y^2=(p-z)(p^{k-1}+p^{k-2}z+\cdots +p z^{k-2}+z^{k-1})$.
If $z\equiv 2k \bmod 4k$, then $p-z \equiv 2k +1 \bmod 4k$.
Since $k$ is odd, $2k+1 \equiv 3 \bmod 4$.
Hence $p-z$ must contain a prime divisor $q \equiv 3 \bmod 4$
with odd multiplicity. Note that $\gcd(q,k)=1$, as
otherwise $q|k$ and $0\equiv p-z\equiv 2k+1\equiv 1 \bmod q$ gives a contradiction.

Recall that by the general classification of integers which are sums of
two squares the integer $x^2+y^2$ contains prime factors
$q \equiv 3 \bmod 4$ with
even multiplicity only. Therefore both $p-z$ and
$p^{k-1}+p^{k-2}z+\cdots +p z^{k-2}+z^{k-1}$
 are
divisible by $q$. With $p \equiv z \bmod q$ it follows that

\[ p^{k-1}+p^{k-2}z+\cdots +p z^{k-2}+z^{k-1}\equiv kz^{k-1} \equiv 0
\bmod q.\]
This implies that $q \mid z$ and hence $q\mid p$, which is impossible,
as $q=p$ would contradict $p\equiv 1 \bmod 4$.

Also note that there are no congruence obstructions
that would imply that in $x^2+y^2+z^k=p^{k}$ there are no 
solutions with $z\equiv 2k \bmod 4k$.

To see this first observe that for a fixed odd prime $q$
one can choose an integer $z \equiv 2k \bmod 4k$
such that $q$ is coprime to $p^k-z^k$; similarly, for $q=8$ just
choose $z=2k$.
For this fixed $z$ the congruence
 $x^2+y^2+z^k \equiv p^k \bmod q$ has a nonsingular
solution in $x$ and $y$ which by Hensel's lemma can be lifted to a $q$-adic or
$2$-adic solution, respectively.

\end{proof}

By the prime
number theorem in arithmetic progressions, the number of such
examples,  $p^{k} \le N$ with $p \equiv 1 \bmod 4k$, is asymptotically
\[  \frac{1}{\varphi(4k)} \int_2^{N^{1/k }} \frac{dt}{\log t} 
\sim \frac{k}{2 \varphi(k)}\frac{N^{1/k}}{\log N}.
\]

\section{Two squares and an even $k$-th-power}

\subsection{Two squares and a $k$-th power, $k \equiv 0 \bmod 4$}

\begin{theorem}{\label{biquadrate}}
Suppose that $4\mid k$ and let $p$ be a prime with $p \equiv 7 \bmod 8$.
Let $n\equiv 1 \bmod 8$ be either $1$ or
consist of prime factors congruent to $1 \bmod 4$ only, and assume
that $n<p$.
Then there are no positive integers $x,y,z$ with $x^2+y^2+z^k=(np)^2$.
\end{theorem}
\begin{proof}
Let $k=2t$, where $t$ is even.
Assume there are solutions, then $x^2+y^2=(np-z^t)(np+z^t)$.
If $z$ is even, then $np-z^t \equiv 3 \bmod 4$. If $z$ is odd, then
$np-z^t\equiv 6 \bmod 8$.
In both cases $np-z^t$ must contain a prime divisor 
$q \equiv 3 \bmod 4$
with odd multiplicity.
Therefore, as in the proof of Theorem \ref{kthpower},
we conclude that both $np-z^t$ and $np+z^t$ are
divisible by $q$.
Hence their sum $2np$ and their difference $-2z^t$ are also divisible 
by $q$.
Since $2n\not\equiv 0 \bmod q$, and since
$p$ is prime: $p=q$, and since $z \neq 0$: $q$ divides $z$.
But this gives a contradiction:
\[x^2+y^2+z^k > q^k\geq q^4>(nq)^2=(np)^2.\]

\end{proof}

Let us give an estimate of the number of integers $np\leq N$, 
with $n\equiv 1\bmod 8$ consisting of prime
factors $1 \bmod 4$ only, and $n<p$.

Recall that by a theorem of Landau the number of integers $n \leq N$
consisting of  prime factors $1 \bmod 4$ only is of order of magnitude
$\frac{N}{(\log N)^{1/2}}$, and about one half of these numbers satisfy the
congruence restriction
 $n \equiv 1 \bmod 8$. 

Let $f: \N\rightarrow \{0,1\}$ be the characteristic 
function of these integers $n$, i.e. we put
$f(n)=1$, if $n\equiv 1\bmod 8$, and all prime factors
  of $n$ are $1\bmod 4$; otherwise we put $f(n)=0$. Now
\[
\sum_{np\leq N, n<p} f(n) = \sum_{n \leq N/p, n<p}f(n)
\gg \sum_{ N^{1/2}\leq p \leq N^{3/4}}\frac{N/p}{(\log (N/p))^{1/2}}\gg
\frac{N}{(\log N)^{1/2}},
\]
where we used that
\[\sum_{ N^{1/2}\leq p \leq N^{3/4}}\frac{1}{p}=
\log \log N^{3/4}-\log \log N^{1/2}+o(1)=\log (3/2)+o(1)\gg 1.\]
(In view of Landau's theorem this order is the right order of magnitude.)
Hence the number of exceptional $(np)^2 \le N$
provided by Theorem \ref{biquadrate} is
$\gg \frac{N^{1/2}}{(\log N)^{1/2}}$.

Note that as for Theorem \ref{kthpower}
one can check that there are no congruence obstructions for the representation of $(np)^2$.

\newpage

\subsection{Two squares and a $k$-th power, $k \equiv 2 \bmod 4$}
\begin{theorem}{\label{k=2t}}
Suppose that $k\equiv 2 \bmod 4$, $k \geq 6$ and 
let $p$ be a prime with $p \equiv 7 \bmod 8$.
Let $n<p$ be an integer either $1$ or consisting of prime factors congruent to 
$1 \bmod 4$ only, and
$n\equiv 1 \bmod 8$.
Then there are no positive integers $x,y,z$, where $2\mid z$,  
with $x^2+y^2+z^k=(np)^2$.
\end{theorem}
\begin{proof}
The proof is almost verbatim as above.
\end{proof}

Let us remark that as above one shows that the number of exceptional
$(np)^2 \le N$
provided by Theorem \ref{k=2t} is $\gg \frac{N^{1/2}}{(\log N)^{1/2}}$.
Further note that in a similar way as for Theorem \ref{kthpower}
one observes that there are no congruence obstructions for the
requested representation of $(np)^2$.

\section{Afterthought}
A major part of the paper was actually written around 2007/8.
We had shown earlier versions of 
this paper to several colleagues, hoping
that someone would write a more detailed explanation based on tools from
arithmetic geometry such as the 
Brauer-Manin obstruction. Indeed, in this way the question has come to Fabian
Gundlach \cite{Gundlach:2013} who was very recently able to give a 
detailed and general account. 

As Gundlach refers to our work as an unpublished manuscript, and as our proofs
use a much less sophisticated language, it seems desirable to have this paper
in final form. The main part of this paper is a slightly improved version, 
compared to the manuscript Gundlach referred to. In particular,
the version cited by Gundlach \cite{Gundlach:2013} had in Theorem \ref{kthpower} the same 
statement and proof with $p^{2k}$ rather than $p^k$. Also in 
Theorem \ref{biquadrate} and \ref{k=2t}
we now have an additional 
factor $n$, thanks to an observation of J.C. Schlage-Puchta.
In other words, the
current version gives slightly stronger results.

\ \\
\noindent 2010 Mathematics Subject Classification:\\
Primary: 11E25\\
\ \\
Secondary: 11P05\\


\end{document}